\documentclass[11pt,a4paper]{article}
\usepackage[lmargin=2cm,rmargin=2cm,bottom=2cm,top=2cm,twoside=False]{geometry}

\usepackage{caption}
\usepackage{subcaption}

\usepackage{fullpage,amssymb,amsmath}
\usepackage{graphicx}
\usepackage{enumerate}
\usepackage{tikz}
\usetikzlibrary{shapes}

\usepackage[T1]{fontenc}

 \usepackage{xcolor}
 \usepackage{mathtools}
\usepackage{microtype}
\usepackage{amsfonts}
\usepackage{comment}
\usepackage[english]{babel}
\usepackage{mathrsfs}
\usepackage[vlined, ruled, linesnumbered]{algorithm2e}
\DontPrintSemicolon

\usepackage{trimspaces}
\usepackage{nccfoots}
\usepackage{setspace}
\usepackage{inconsolata}
\usepackage{libertine}
\usepackage[absolute]{textpos}
\usepackage{thmtools}
\usepackage{thm-restate}
\usepackage{authblk}

\usepackage{todonotes}

\usepackage{longtable}

\definecolor{blue}{rgb}{0.1,0.2,0.5}
\definecolor{brown}{rgb}{0.6,0.6,0.2}
\usepackage[ocgcolorlinks, linkcolor={blue}, citecolor={brown}]{hyperref}

\usepackage{comment}

\usepackage[amsmath,thmmarks,hyperref]{ntheorem}
\usepackage{cleveref}

\crefformat{page}{#2page~#1#3}%
\Crefformat{page}{#2Page~#1#3}%
\crefformat{equation}{#2(#1)#3}%
\Crefformat{equation}{#2(#1)#3}%
\crefformat{figure}{#2Figure~#1#3}%
\Crefformat{figure}{#2Figure~#1#3}%
\crefformat{section}{#2Section~#1#3}
\Crefformat{section}{#2Section~#1#3}
\crefformat{chapter}{#2Chapter~#1#3}
\Crefformat{chapter}{#2Chapter~#1#3}
\crefformat{chapter*}{#2Chapter~#1#3}
\Crefformat{chapter*}{#2Chapter~#1#3}
\crefformat{part}{#2Part~#1#3}
\Crefformat{part}{#2Part~#1#3}
\crefformat{enumi}{#2(#1)#3}
\Crefformat{enumi}{#2(#1)#3}

\crefname{claim}{Claim}{Claims}
\Crefname{claim}{Claim}{Claims}

\usepackage{enumerate}

\usepackage{latexsym}

\theoremnumbering{arabic}
\theoremstyle{plain}
\theoremsymbol{}
\theorembodyfont{\itshape}
\theoremheaderfont{\normalfont\bfseries}
\theoremseparator{.}

\newtheorem{theorem}{Theorem}
\crefformat{theorem}{#2Theorem~#1#3}
\Crefformat{theorem}{#2Theorem~#1#3}

\newcommand{\newtheoremwithcrefformat}[2]{%
  \newtheorem{#1}[theorem]{#2}%
  \crefformat{#1}{##2\MakeUppercase#1~##1##3}%
  \Crefformat{#1}{##2\MakeUppercase#1~##1##3}%
}
\newcommand{\newseptheoremwithcrefformat}[2]{%
  \newtheorem{#1}{#2}%
  \crefformat{#1}{##2\MakeUppercase#1~##1##3}%
  \Crefformat{#1}{##2\MakeUppercase#1~##1##3}%
}

\newtheoremwithcrefformat{lemma}{Lemma}
\newtheoremwithcrefformat{proposition}{Proposition}
\newtheoremwithcrefformat{observation}{Observation}
\newtheoremwithcrefformat{conjecture}{Conjecture}
\newtheoremwithcrefformat{corollary}{Corollary}
\newseptheoremwithcrefformat{claim}{Claim}
\theorembodyfont{\upshape}
\newtheoremwithcrefformat{example}{Example}
\newtheoremwithcrefformat{remark}{Remark}
\newseptheoremwithcrefformat{definition}{Definition}
\newseptheoremwithcrefformat{question}{Question}

\theoremstyle{nonumberplain}
\theoremheaderfont{\scshape}
\theorembodyfont{\normalfont}
\theoremsymbol{\ensuremath{\square}}
\newtheorem{proof}{Proof}

\theoremsymbol{\ensuremath{\lrcorner}}
\newtheorem{claimproof}{Proof of Claim}

\def\cqedsymbol{\ifmmode$\lrcorner$\else{\unskip\nobreak\hfil
\penalty50\hskip1em\null\nobreak\hfil$\lrcorner$
\parfillskip=0pt\finalhyphendemerits=0\endgraf}\fi}

\tikzset{
    position/.style args={#1:#2 from #3}{
        at=(#3.#1), anchor=#1+180, shift=(#1:#2)
    }
}

\newcommand{\Oh}{\mathcal{O}}

\let\originalleft\left
\let\originalright\right
\renewcommand{\left}{\mathopen{}\mathclose\bgroup\originalleft}
\renewcommand{\right}{\aftergroup\egroup\originalright}

\renewcommand{\leq}{\leqslant}
\renewcommand{\geq}{\geqslant}
\renewcommand{\tilde}{\widetilde}

\newcommand{\comp}{\mathsf{CC}}

\usepackage{todonotes}

\usepackage{lineno}

\newcommand\calD{\ensuremath{\mathcal{D}}}

\newcommand\tB{\tilde{B}}

\newcommand\card[1]{|#1|}

\begin{document}

\author[1]{Jakub Gajarsk{\' y}}
\author[1,2]{Lars Jaffke}
\author[3]{Paloma T.\ Lima}
\author[1]{Jana Novotn{\' a}}
\author[1]{Marcin Pilipczuk}
\author[1,5]{Pawe{\l}~Rz{\k a}\.{z}ewski}
\author[1,4]{U{\' e}verton S.\ Souza}
\affil[1]{University of Warsaw, Poland}
\affil[2]{University of Bergen, Norway}
\affil[3]{IT University of Copenhagen, Denmark}
\affil[4]{Universidade Federal Fluminense, Niter\'{o}i, Brazil}
\affil[5]{Warsaw University of Technology, Poland}

\title{Taming graphs with no large creatures and skinny ladders%
\thanks{This research is a part of a project that has received funding from the European Research Council (ERC) under the European Union's Horizon 2020 research and innovation programme Grant Agreement 714704.}
}
\date{}
\maketitle

\begin{textblock}{20}(0,12.75)
\includegraphics[width=40px]{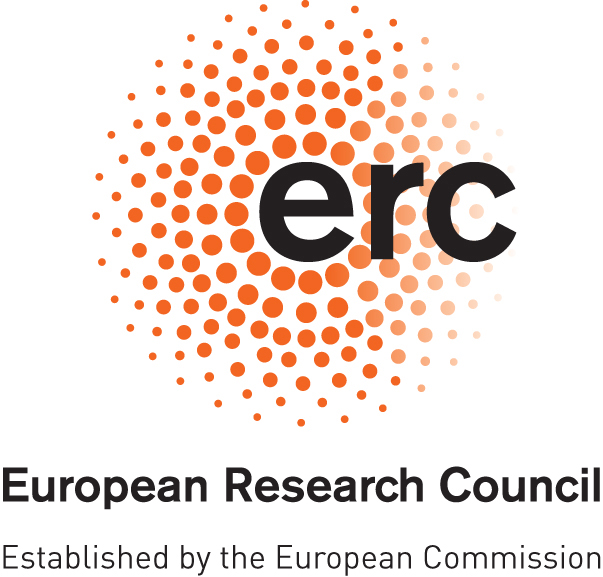}%
\end{textblock}
\begin{textblock}{20}(0,13.59)
\includegraphics[width=40px]{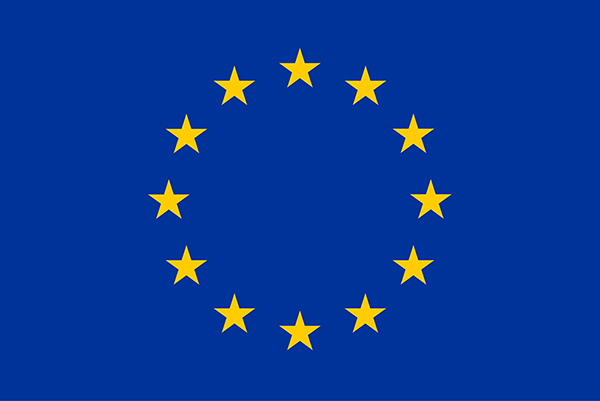}%
\end{textblock}

\begin{abstract}
We confirm a conjecture of Gartland and Lokshtanov [arXiv:2007.08761]: 
if for a hereditary graph class $\mathcal{G}$ there exists a constant $k$
such that no member of $\mathcal{G}$ contains a $k$-creature as an induced subgraph
or a $k$-skinny-ladder as an induced minor, then there exists a polynomial $p$
such that every $G \in \mathcal{G}$ contains at most $p(|V(G)|)$ minimal separators.
By a result of Fomin, Todinca, and Villanger~[SIAM J. Comput. 2015] the latter
entails the existence of polynomial-time algorithms for \textsc{Maximum Weight Independent Set},
\textsc{Feedback Vertex Set} and many other problems, when restricted to an input
graph from $\mathcal{G}$.
Furthermore, as shown by Gartland and Lokshtanov, our result implies a full dichotomy
of hereditary graph classes defined by a finite set of forbidden induced subgraphs
into tame (admitting a polynomial bound of the number of minimal separators)
and feral (containing infinitely many graphs with exponential number of minimal separators). 
\end{abstract}

\section{Introduction}
For a graph $G$, a set $S \subseteq V(G)$ is a \emph{minimal separator}
if there are at least two connected components $A,B$ of $G-S$ with $N(A) = N(B) = S$
(so that $S$ is an inclusion-wise minimal set that separates a vertex of $A$ from a vertex of $B$).
Around the year 2000, Bouchitt\'{e} and Todinca presented a theory of minimal separators and
related objects called \emph{potential maximal cliques} and showed their usefulness for providing
efficient algorithms~\cite{BouchitteT01}.
In particular, the \textsc{Maximum Weight Independent Set} problem (given a vertex-weighted graph,
find a subset of pairwise nonadjacent vertices of maximum total weight)
can be solved in time bounded polynomially in the size and the number of minimal separators in
the graph. 
This result has been generalized by Fomin, Todinca, and Villanger to a large
range of problems that can be defined as finding an induced subgraph of constant treewidth
with some \textsf{CMSO}$_2$-expressible property~\cite{FominTV15}; this includes, for example, \textsc{Longest Induced Path} or \textsc{Max Induced Forest}, which is by complementation equivalent to \textsc{Feedback Vertex Set}. 

When do these metaalgorithmic results give efficient algorithms?
In other words, which restrictions on graphs guarantee a small number of minimal separators?
On one hand, it is easy to see that an $n$-vertex chordal graph has $\Oh(n)$ minimal separators.
On the other hand, consider the following two negative examples. 
For $k \geq 3$, the \emph{$(k,1)$-prism} consists of two $k$-vertex cliques with vertex sets $X = \{x_1,\ldots,x_k\}$
and $Y = \{y_1,\ldots,y_k\}$ and a perfect matching $\{x_iy_i~|~i \in [k]\}$. 
It is easy to see that the $(k,1)$-prism has $2^k-2$ minimal separators: any choice
of one endpoint of each edge $x_iy_i$ gives a minimal separator, except for the choices $X$ and $Y$.
The \emph{$(k,3)$-theta} consists of $k$ independent edges $\{x_iy_i~|~i \in [k]\}$,
a vertex $x$ adjacent to all vertices $x_i$ and a vertex $y$ adjacent to all vertices
$y_i$ (the intuition behind the notation is that the graph consists of $k$ paths of length 3, joining $x$ and $y$). Again, any choice of one endpoint of each edge $x_iy_i$ gives a minimal separator.
Thus, both the $(k,1)$-prism and the $(k,3)$-theta have an exponential (in the number of vertices)
number of minimal separators. 

In 2019, Milani\v{c} and Piva\v{c} initiated a systematic study of the question which graph
classes admit a small bound on the number of minimal separators in its members~\cite{MilanicP19,MilanicP21}. 
A graph class $\mathcal{G}$ is \emph{tame} if there exists a polynomial $p_{\mathcal{G}}$
such that for every $G \in \mathcal{G}$ the number of minimal separators of $G$
is bounded by $p_{\mathcal{G}}(|V(G)|)$. Clearly, if $\mathcal{G}$ is tame, 
then \textsc{Maximum Weight Independent Set} and all problems captured by the formalism
of~\cite{FominTV15} are solvable in polynomial time when the input graph comes from $\mathcal{G}$.
On the opposite side of the spectrum, $\mathcal{G}$ is \emph{feral} if there exists $c > 1$
such that for infinitely many graphs $G \in \mathcal{G}$ it holds that $G$ has at least $c^{|V(G)|}$
minimal separators. Following the previous examples, the class of chordal graphs is tame
while the class of all $(k,1)$-prisms and/or all $(k,3)$-thetas (over all $k$) is feral.
Milani\v{c} and Piva\v{c} provided a full tame/feral dichotomy for hereditary graph classes (i.e., closed under vertex deletion) defined by minimal forbidded induced subgraphs on at most $4$ vertices~\cite{MilanicP19,MilanicP21}.

A subsequent work of Abrishami, Chudnovsky, Dibek, Thomass\'{e}, Trotignon, and Vuskovi\v{c}~\cite{AbrishamiCDTTV22} indicated that the main line of distinction between tame and feral graph classes
should lie around the notion of a \emph{$k$-creature}.
A $k$-creature in a graph $G$ is a tuple $(A,B,X,Y)$ of pairwise disjoint nonempty vertex sets 
such that (i) $A$ and $B$ are connected, (ii) $A$ is anti-adjacent to $Y \cup B$ and $B$ is anti-adjacent to $A \cup X$, (iii) every $x \in X$ has a neighbor in $A$ and every $y \in Y$ has a neighbor in $B$;
(iv) $|X|=|Y|=k$ and $X$ and $Y$ can be enumerated as $X = \{x_1,\ldots,x_k\}$, $Y = \{y_1,\ldots,y_k\}$ such that $x_iy_j \in E(G)$ if and only if $i=j$. 
We say that $G$ is \emph{$k$-creature-free} if $G$ does not contain a $k$-creature as an induced subgraph.
Similarly as in the examples of the $(k,1)$-prism and the $(k,3)$-theta, any choice of one endpoint
of every edge $x_iy_i$ gives a minimal separator in the subgraph induced by the creature
(which, in turn, can be easily lifted to a minimal separator in $G$).
Hence, if $G$ contains a $k$-creature as an induced subgraph, it contains at least $2^k$ minimal separators.
In fact, the notion of a $k$-creature is a common generalization of the examples of the $(k,1)$-prism
and the $(k,3)$-theta.
Indeed, the $(k,3)$-theta contains a $k$-creature with $A = \{x\}$ and $B = \{y\}$
while the $(k,1)$-prism contains a $(k-2)$-creature with $A = \{x_{k-1}\}$, $B = \{y_k\}$, 
$X= \{x_1,\ldots,x_{k-2}\}$, and $Y = \{y_1,\ldots,y_{k-2}\}$. 
In particular, Abrishami et al. conjectured that if for a hereditary graph class $\mathcal{G}$
there exists $k$ such that no $G \in \mathcal{G}$ contains a $k$-creature as an induced subgraph,
then $\mathcal{G}$ is tame.
(Observe that a presence of arbitrarily large creatures in a hereditary graph class
 does not immediately imply that the graph class is feral, as the sets $A$ and $B$ can be of
 superpolynomial size in $k$.)

A counterexample to the conjecture of~\cite{AbrishamiCDTTV22}
has been provided by Gartland and Lokshtanov
in the form of a \emph{$k$-twisted ladder}~\cite{GL20}. 
They observed that, despite the fact that the conjecture of~\cite{AbrishamiCDTTV22} is false,
every example they can construct ``looks like a twisted ladder'', which indicates
that the tame/feral boundary
for hereditary graph classes should not be far from the said conjecture. 
To support this intuition, they introduced the notion of a \emph{$k$-skinny ladder}
(a graph consisting of two induced antiadjacent paths $P = (p_1,\ldots,p_k)$, $Q=(q_1,\ldots,q_k)$,
and independent set $R = (r_1,\ldots,r_k)$, and edges $\{p_ir_i,q_ir_i~|~i\in[k]\}$),
  noted that a $k$-skinny-ladder is an induced minor of every counterexample they constructed, 
and proved the following.
\begin{theorem}\label{thm:GL20}
For every $k$ there exists a constant $c_k$
such that if a graph $G$ is $k$-creature-free
and does not contain a $k$-skinny-ladder
as an induced minor, then the number of minimal separators in $G$
is bounded by $c_k |V(G)|^{c_k \log |V(G)|}$, that is, quasi-polynomially in 
the size of $G$.
\end{theorem}

Gartland and Lokshtanov conjectured that this dependency should be in fact polynomial.
Our main result of this paper is a proof of this conjecture.
\begin{theorem}\label{thm:conj}
For every $k \in \mathbb{N}$ there exists a polynomial $\mathrm{q}$ of degree
$\Oh(k^3 \cdot (8k^2)^{k+2})$ such that every graph $G$
that is $k$-creature-free and does not contain $k$-skinny-ladder as an induced minor
contains at most $\mathrm{q}(|V(G)|)$ minimal separators.
\end{theorem}
That is, every hereditary graph class $\mathcal{G}$ for which there exists $k$
such that no member of $G$ contains a $k$-creature nor $k$-skinny-ladder as an induced minor,
is tame.

As proven in~\cite{GL20}, \cref{thm:conj} implies a dichotomy result into tame and feral
graph classes for all hereditary graph classes defined by a finite list of forbidden induced subgraphs.
(For the exact definitions of graphs in the statement, we refer to~\cite{GL20}.)
\begin{theorem}\label{thm:dich}
Let $\mathcal{G}$ be a graph class defined by a finite number of forbidden induced subgraphs.
If there exists a natural number $k$ such that $\mathcal{G}$ does not contain all
$k$-theta, $k$-prism, $k$-pyramid, $k$-ladder-theta, $k$-ladder-prism, $k$-claw, and $k$-paw graphs,
  then $\mathcal{G}$ is tame. Otherwise $\mathcal{G}$ is feral.
\end{theorem}

Our proof builds upon the proof of \cref{thm:GL20} of~\cite{GL20}
and provides a new way of analysing one of the core invariants. 
For a graph $G$ and a set $S$, define
\[
\zeta_G(S) = \max \{ |I| ~:~ I \subseteq S \text{ is an independent set and for every }  v \notin S \text{ we have } |N(v) \cap I| \leq 1 \}.
\]
That is, we want a set $I \subseteq S$ of maximum possible size that is not only independent,
but no vertex outside $S$ is adjacent to more than one vertex of $I$.  
In the proof of \cref{thm:GL20} of~\cite{GL20}, an important step is to prove that 
a minimal separator $S$ with huge $\zeta_G(S)$ gives rise to a large skinny ladder
as an induced minor. 
Our main technical contribution is an improved way of analysing minimal separators $S$ with small $\zeta_G(S)$.

\begin{theorem}\label{thm:main}
For every $k,L \in \mathbb{N}$ there exists a polynomial $\mathrm{p}$ of degree $\Oh(k^3 \cdot L)$, such that the following holds.
For every $k$-creature-free graph $G$, the number of minimal separators $S$ satisfying $\zeta_G(S) \leq L$ is at most $\mathrm{p}(|V(G)|)$.
\end{theorem}

After brief preliminaries in \cref{sec:prelims}, we prove \cref{thm:main} in \cref{sec:main}. 
We show how \cref{thm:main} implies \cref{thm:conj} (with the help of
some tools from~\cite{GL20}) in \cref{sec:wrap-up}.

\section{Preliminaries}\label{sec:prelims}
Let $G$ be a graph, $v$ be a vertex of $G$, and $S$ be a subset of vertices.
By $N_G(v)$ we denote the set of neighbors of $v$. Similarly, by $N_G(S)$ we denote the set $\bigcup_{x \in S} N_G(x) \setminus S$. If the graph $G$ is clear from the context, we simply write $N(v)$ and $N(S)$.

For sets $A,B,C$, whenever we write $A \setminus B \setminus C$, the set difference operation associates from the left, meaning that $A \setminus B \setminus C$ is equivalent to $(A \setminus B) \setminus C$ (and, alternatively, to $A \setminus (B \cup C)$).

By $G - S$ we denote the graph obtained from $G$ by deleting all vertices from $S$ along with incident edges,
and by $G[S]$ we denote the graph induced by the set $S$, i.e., $G - (V(G) \setminus S)$.
By $\comp(G)$ we denote the set of connected components of $G$, given as vertex sets.

A \emph{matching} in $G$ is a set of pairwise disjoint edges.
We say that a matching $\{x_i y_i \mid i \in [k]\}$ is a \emph{semi-induced matching between $\{x_1, \ldots, x_k\}$ and $\{y_1, \ldots, y_k\}$}
if for all $i, j \in [k]$, $x_iy_j \in E(G)$ if and only if $i = j$.

For vertices $u,v$, a set $S \subseteq V(G) \setminus \{u,v\}$ is a \emph{$u$-$v$-separator} if $u$ and $v$ are in different connected components of $G - S$. We say that $S$ is a \emph{minimal $u$-$v$-separator} if it is a $u$-$v$-separator
and no proper subset of $S$ is a $u$-$v$-separator. A set $S$ is a \emph{minimal separator} if it is a minimal $u$-$v$-separator for some $u,v$.
Equivalently, $S$ is a minimal separator if there are at least two components $A,B \in \comp(G - S)$ such that $N(A) = N(B) = S$. 
Any component $A \in \comp(G - S)$ with 
$N(A) = S$ is called \emph{full to $S$}; a minimal separator has at least two full components. 

We define
\[
S_G^v = \{N(v) \cap S : v \notin S \text{ and } S \text{ is a minimal separator of } G\}.
\]
The following result of Gartland and Lokshtanov will be a crucial tool used in our argument.

\begin{lemma}[Gartland and Lokshtanov~\cite{GL20}]\label{lem:GL:Sv}
If $G$ is a $k$-creature-free graph, then for every $v \in V(G)$ it holds that $|S_G^v| \leq |V(G)|^{k+1}$.
\end{lemma}

Let us also recall the crucial definition. For a set $S \subseteq V(G)$ we define
\[
\zeta_G(S) = \max \{ |I| ~:~ I \subseteq S \text{ is an independent set and for every }  v \notin S \text{ we have } |N(v) \cap I| \leq 1 \}.
\]

\section{Proof of \cref{thm:main}}\label{sec:main}
    We prove the theorem by induction on $L$
    with the exact bound of
    $n^{L(4+(k^2+2)(k+2))}$ minimal separators.

    Note that if $S\neq\emptyset$, then $\zeta_G(S)\geq 1$, since for any $u\in S$, the set $I=\{u\}$ satisfies the required properties.
    Thus, in the base case, when $L = 0$, the only candidate for $S$ is the empty set, therefore the claim holds vacuously. Also, the claim is immediate for $n=1$, so we assume $n > 1$.
    
    Let $S$ be a minimal separator of $G$, 
    and let $A$ and $B$ be two connected components of $G - S$ that are full to $S$.
    If there is a vertex $v \in V(G) \setminus S$ such that 
    $N(v) \supseteq S$, then $S \in S_G^v$. There are at most $n^{k+2}$ such separators $S$ by \cref{lem:GL:Sv};
    we may therefore assume that no such vertex exists.
    Let $\tB$ be a minimal connected subset of $B$ that still dominates $S$, i.e.\ such that $N(\tB) \supseteq S$.
    Let $u \in \tB$ be such that $\tB \setminus \{u\}$ is still connected.
    Such a vertex $u$ can be found, for instance, as a leaf of a spanning tree of $\tB$.
    We define the following sets that will be important throughout the proof,
    see \cref{fig:ASB:overview}.
    
    \newcommand\asbheight{.3\textheight}
    \begin{figure}
        \centering
        \includegraphics[height=\asbheight]{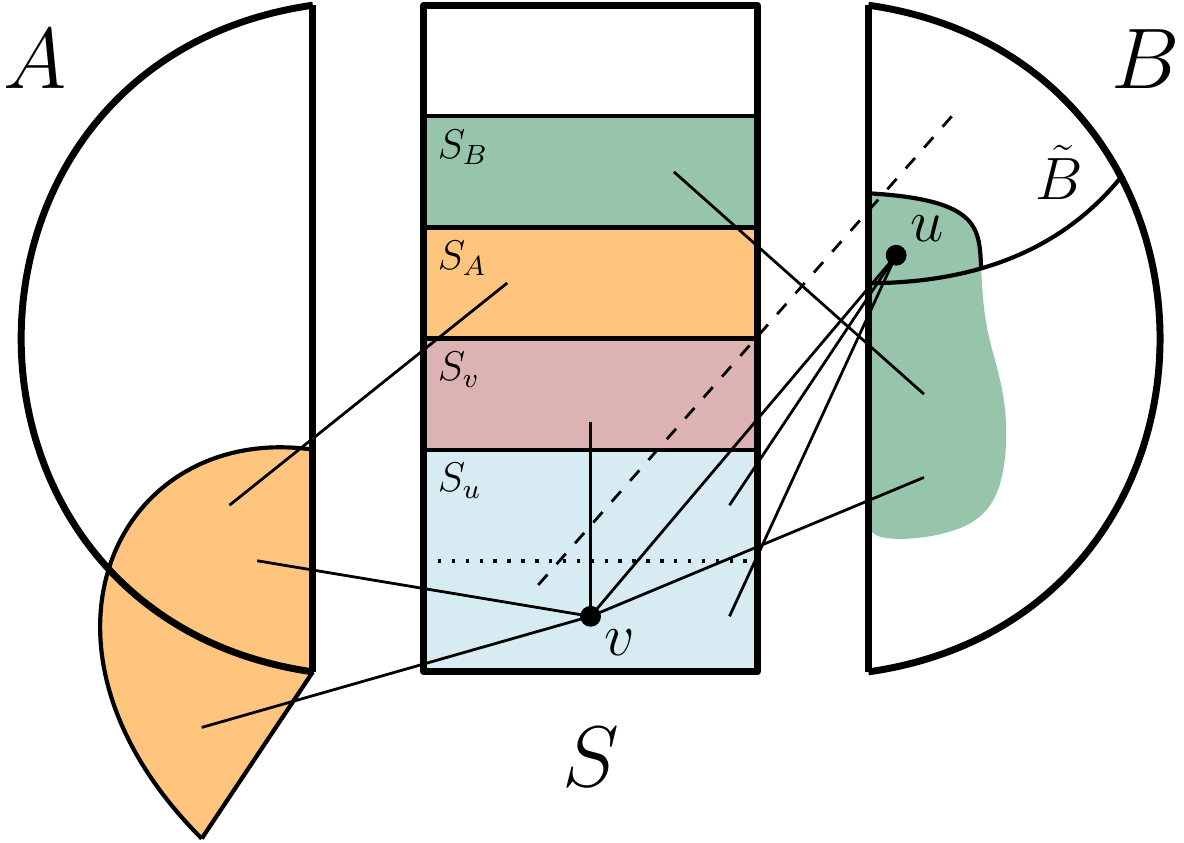}
        \caption{Subsets of $S$ defined in the proof of~\cref{thm:main}. The full lines indicate adjacencies. The dotted line inside $S_u$ indicates a partition of $S_u$ between the private neighbors of $u$ (below) and other neighbors of~$u$ (above). The dashed line indicates there is no edge between the sets.}    
        \label{fig:ASB:overview}
    \end{figure}

    \begin{itemize}
        \item We let $v \in S \cap N(u) \setminus N(\tB \setminus \{u\})$. In words, $v$ is a private neighbor (with respect to $\tB$) of $u$ in $S$. Such a vertex $v$ exists by the minimality of $\tB$.
        \item $S_u = N(u) \cap S$.
        \item $S_v = (N(v) \cap S) \setminus S_u$.
        \item $S_A = (S \setminus S_u \setminus S_v) \cap N(N(v) \setminus S \setminus B)$. That is, $S_A$ contains the vertices of $S\setminus S_u \setminus S_v$ that have a common neighbor with $v$ in $N(v)\setminus S\setminus B$. 
        \item $S_B = (S \setminus S_u \setminus S_v \setminus S_A) \cap N(N(v) \cap B)$. Similarly, $S_B$ contains the vertices of $S\setminus S_u \setminus S_v \setminus S_A$ that have a common neighbor with $v$ in $N(v)\cap B$.
    \end{itemize}
    
    Our goal is now to identify a small set that dominates $S^* = S_u \cup S_v \cup S_A \cup S_B$.
    We will repeatedly use~\cref{lem:GL:Sv} on the vertices of this set in order to bound 
    the number of choices for~$S^*$.
    We then show that we can find a minimal separator $S_0$ in $S \setminus S^*$ such that $A$ is a full component in $G - (S^* \cup S_0)$ and there is a component containing $\tB \setminus \{u\}$. 
    We will be able to show that $\zeta_{G - S^*}(S_0) < \zeta_G(S)$ which allows us to conclude using the induction hypothesis on $G-S^*$.
    
    \begin{figure}
        \centering
        \begin{subfigure}{.3\textwidth}
            \includegraphics[width=\textwidth]{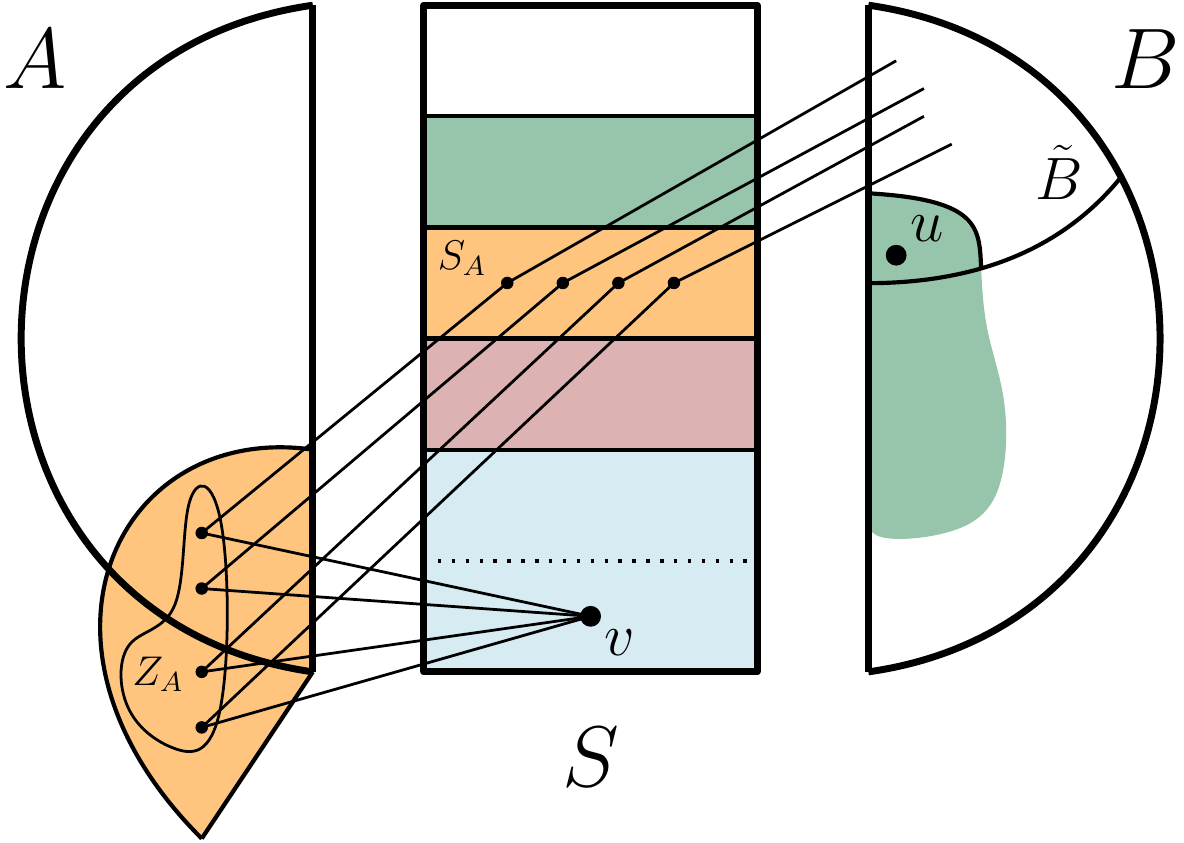}
            \caption{$|Z_A|$-creature obtained in the proof of~\cref{claim:ZA}.}
            \label{fig:cr:ZA}
        \end{subfigure}
        \hfill
        \begin{subfigure}{.3\textwidth}
            \includegraphics[width=\textwidth]{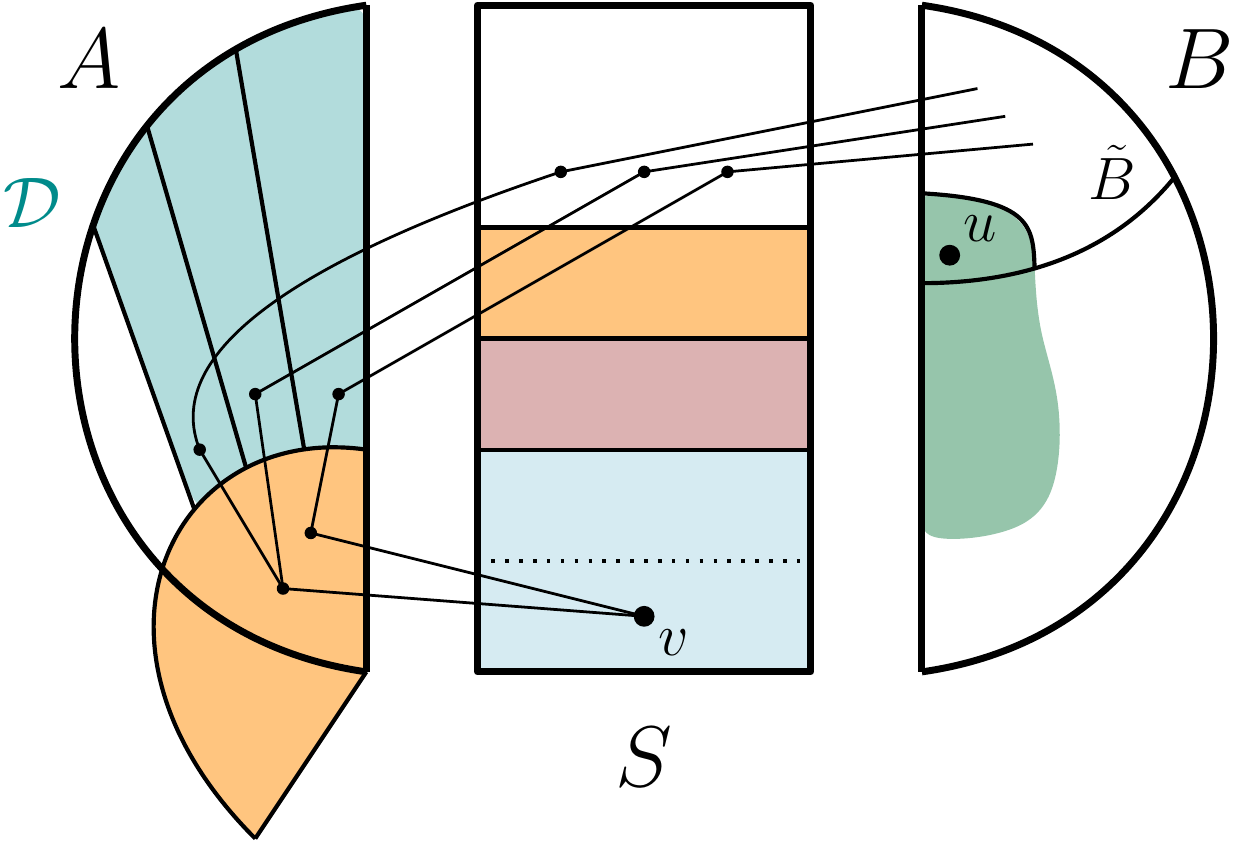}
            \caption{$|\calD|$-creature obtained in the proof of~\cref{claim:calD}.}
            \label{fig:cr:calD}
        \end{subfigure}
        \hfill
        \begin{subfigure}{.3\textwidth}
            \includegraphics[width=\textwidth]{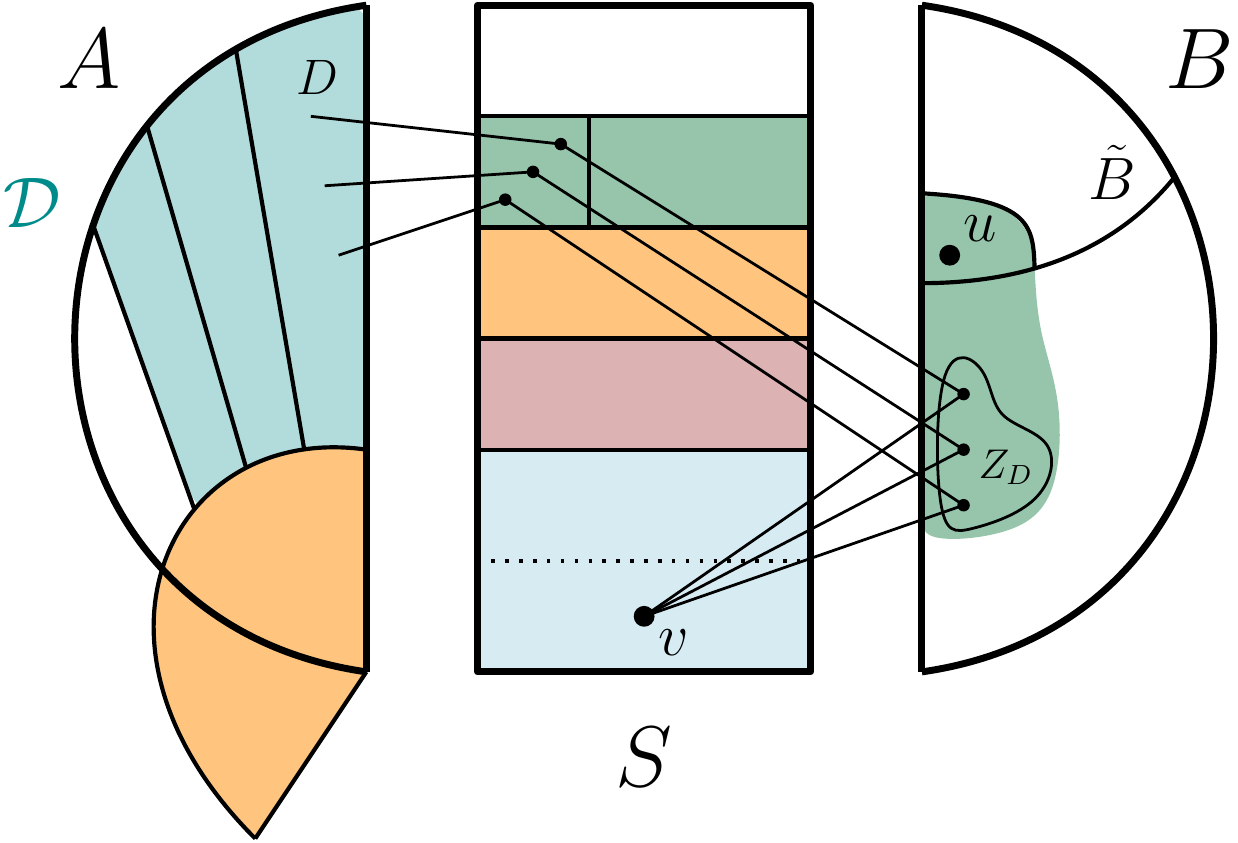}
            \caption{$|Z_D|$-creature obtained in the proof of~\cref{claim:ZD}.}
            \label{fig:cr:ZD}
        \end{subfigure}
        \caption{The creatures of~\cref{thm:main}.}
        \label{fig:creatures}
    \end{figure}
    
    \begin{claim}\label{claim:ZA}
        Let $Z_A \subseteq N(v) \setminus S \setminus B$ be a minimal set such that $N(Z_A) \supseteq S_A$.
        Then $G$ contains a $\card{Z_A}$-creature.
    \end{claim}
    \begin{claimproof}
        By the minimality of $Z_A$, each vertex of $Z_A$ has a private neighbor in $S_A$. Hence,
        there is a semi-induced matching between $Z_A$ and a size-$\card{Z_A}$ subset of $S_A$, say $S_{A, Z}$.
        We obtain the $\card{Z_A}$-creature by 
        considering the sets $(\{v\}, \tB\setminus\{u\}, Z_A, S_{A,Z})$, 
        see \cref{fig:cr:ZA}.
        Indeed, note that by our choice of $u$, we have that $G[\tB\setminus u]$ is connected;
        $v$ has no neighbors in $\tB\setminus \{u\}$, as it is a private neighbor of $u$; $v$ has no neighbors in $S_{A,Z}$ since $S_{A,Z}\subseteq S_A\subseteq S\setminus S_u\setminus S_v$; clearly there are no edges between $\tB\setminus\{u\}$ and $Z_A$ since $S$ is a separator;
$v$ dominates $Z_A$ and $\tB\setminus u$ dominates $S_{A,Z}$.      
    \end{claimproof}
    
    \begin{claim}\label{claim:calD}
        Let $\calD \subseteq \comp(G[A \setminus N(v)])$ be a minimal set of such components that dominates $S \setminus S_u \setminus S_v \setminus S_A$.
        Then $G$ contains a $\card{\calD}$-creature.
    \end{claim}
    \begin{claimproof}
        By the minimality of $\calD$, for each $D\in \calD$ there exists a vertex $y_D\in S\setminus S_u\setminus S_v\setminus S_A$ that is dominated only by vertices of $D$. Let $x_D$ be a vertex of $D$ that is adjacent to $y_D$
        and that is closest to $N(v) \cap A$ in $G[A]$. 
        Note that the edges $\{x_Dy_D~:~D\in\calD\}$ form a semi-induced matching 
		between $\{x_D~\colon~D \in \calD\}$ and $\{y_D~\colon~D \in \calD\}$ in $G$. 
        Let $P_D$ be the set of internal vertices on a shortest path between $x_D$ and $N(v) \cap A$ via $G[A]$. Note that $P_D\subset D$ and that $P_D$ is anti-adjacent to $\bigcup_{D \in \calD} \{y_D\}$. Indeed, the vertices of $P_D$ are not adjacent to $y_D$, as this would contradict the minimality of the distance between $x_D$ and $N(v)$; and for any $D'\neq D$, the vertices of $P_D$ are not adjacent to $y_{D'}$ as this vertex is only dominated by vertices of $D'$, by our choice of $y_{D'}$.
        Then we obtain a $\card{\calD}$-creature by considering the sets
        $(\{v\}\cup (N(v)\cap A) \cup (\bigcup_{D \in \calD} P_D), \tB \setminus \{u\}, \bigcup_{D \in \calD} \{x_D\}, \bigcup_{D \in \calD} \{y_D\})$,
        see \cref{fig:cr:calD}.
        Indeed, note that $G[\{v\}\cup (N(v)\cap A) \cup (\bigcup_{D \in \calD} P_D)]$ and $G[\tB \setminus \{u\}]$ are connected; there are no edges between $\{v\}\cup(N(v)\cap A)$ and $\bigcup_{D \in \calD} \{y_D\}$ since $(\bigcup_{D \in \calD} \{y_D\}) \cap (S_A\cup S_v\cup S_u)=\emptyset$, neither edges between $\bigcup_{D \in \calD} P_D$ and $\bigcup_{D \in \calD} \{y_D\}$ as mentioned above. 
         Note also that $\{v\}\cup(N(v)\cap A)\cup (\bigcup_{D \in \calD} P_D) \cup (\bigcup_{D \in \calD} \{x_D\})$ is anti-adjacent to $\tB \setminus \{u\}$ as argued in the proof of~\cref{claim:ZA}. Finally, note that $(N(v)\cap A) \cup (\bigcup_{D \in \calD} P_D)$ dominates $\bigcup_{D \in \calD} \{x_D\}$, since every $x_D$ either has a neighbor in $P_D$, or in $N(v)\cap A$ if $P_D=\emptyset$.
        Finally, it is easy to see that $\tB \setminus \{u\}$ dominates $\bigcup_{D \in \calD} \{y_D\}$. 
    \end{claimproof}
    
    \begin{claim}\label{claim:ZD}
        Let $\calD$ be as in Claim~\ref{claim:calD}.
        For each $D \in \calD$, let $Z_D \subseteq N(v) \cap B$ be a minimal set that dominates $N(D) \cap S_B$.
        Then $G$ contains a $\card{Z_D}$-creature.
    \end{claim}
    \begin{claimproof}
        By the minimality of $Z_D$, there is a semi-induced matching between $Z_D$ and a size-$\card{Z_D}$ subset $S_{B, Z}$ of $N(D) \cap S_B$.
        We obtain a creature by considering the sets
        $(\{v\}, D, Z_D, S_{B,Z})$, see \cref{fig:cr:ZD}.
        Indeed,
        note that $D$ is connected by definition; 
        $v$ is not adjacent to $S_{B,Z}$ since $S_{B,Z}\cap (S_u\cup S_v)=\emptyset$ and $v$ is not adjacent to $D$ since $D$ is a connected component of $A\setminus N(v)$; $D$ is not adjacent to $Z_D$ since $S$ is a separator; $v$ dominates $Z_D$ by definition of $Z_D$ and $S_{B,Z}\subseteq N(D)$.  
    \end{claimproof}
    
    \begin{figure}
        \centering
        \includegraphics[height=\asbheight]{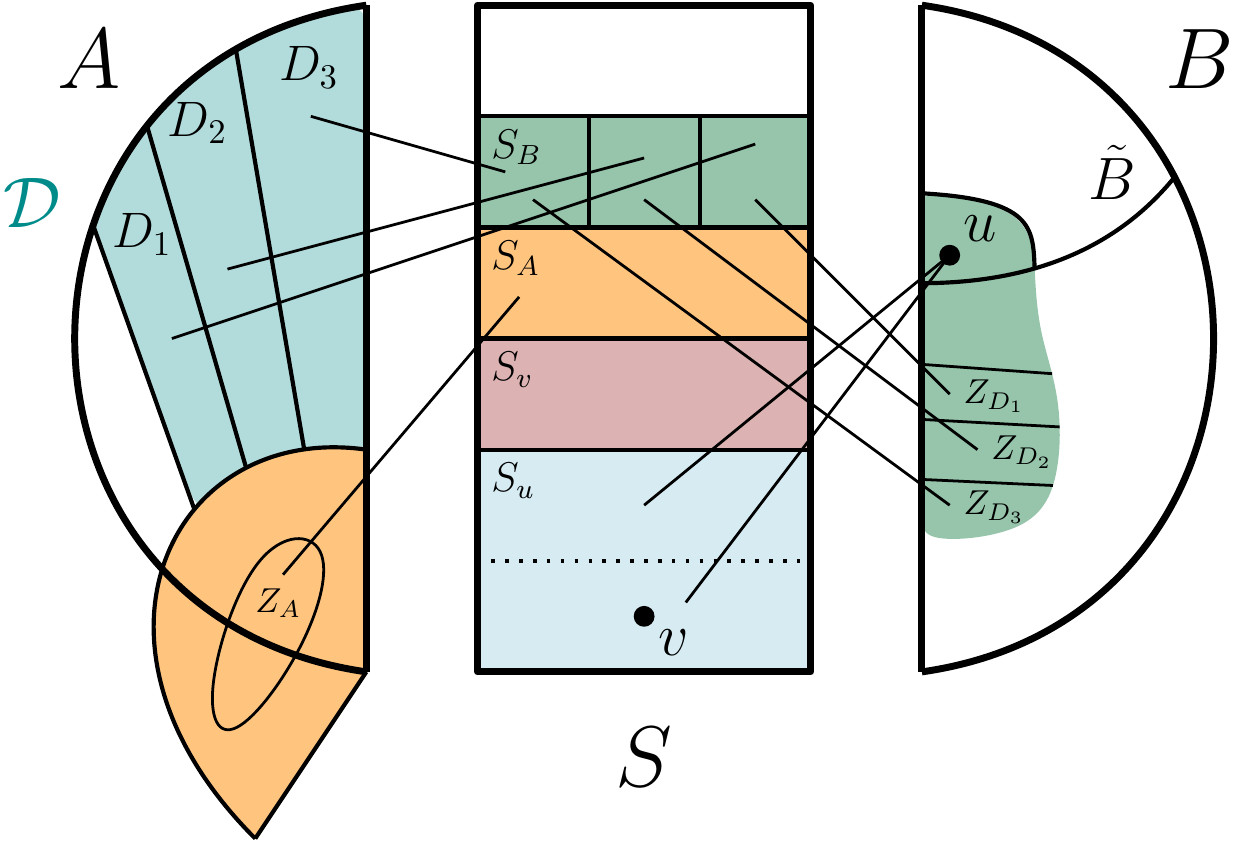}
        \caption{Illustration of how $Q$ is obtained in the proof of~\cref{thm:main}.}
        \label{fig:ASB:Q}
    \end{figure}
    
    Let $Z = \{u\} \cup Z_A \cup \bigcup_{D \in \calD} Z_D$,
    where $Z_A$, $\calD$, and $Z_D$ for $D \in \calD$
    are as defined in~\cref{claim:ZA,claim:calD,claim:ZD}, respectively.
    For all $z \in Z$, let $Q_z = N(z) \cap S$.
    Let $Q = \bigcup_{z \in Z} Q_z$.
    Note that $Q$ contains $S_u$ since $u \in Z$,
    that $Q$ contains $S_A$ since $Z_A \subseteq Z$, 
    and that $Q$ contains $S_B$.
    The latter is due to the fact that the vertices in $\calD$ dominate $S_B$ by choice, and each $Z_D$ where $D \in \calD$ dominates $N(D) \cap S_B$. 
    We illustrate this situation in \cref{fig:ASB:Q}.
    It remains to get a grip on $S_v$.
    
    To do so, let $S' = (S \setminus \{v\}) \cup (N(v) \cap B)$, 
    and note that $S'$ separates $A \cup \{v\}$ from $\tB \setminus \{u\}$.
    Let $S''$ be a minimal subset of $S'$ that still separates $A \cup \{v\}$ from $\tB \setminus \{u\}$.
    Note that there are components $A'' \supseteq A \cup \{v\}$ and $B'' \supseteq \tB \setminus \{u\}$ that are full to $S''$ and $S''$ is a minimal separator.
    Now let $R = N(v) \cap S'' \in S^v_G$, so there are at most $n^{k+2}$ choices for $R$, by~\cref{lem:GL:Sv}.
    We observe that $R \supseteq S_v$, which is due to the fact that $\tB \setminus \{u\}$ dominates $S_v$, and that $S''$ separates $\{v\}$ from $\tB \setminus \{u\}$.
    
    \begin{claim}\label{claim:choicesQR}
        There are at most $n^{k^2(k+2)}$ choices for $Q$,
        and at most $n^{k+2}$ choices for $R$.
    \end{claim}
    \begin{claimproof}
        We already observed the second statement of the claim above. 
        For the first statement,
        by \cref{claim:ZA,claim:calD,claim:ZD} we know that $\card{Z} < k^2$,
        so there are at most $n^{k^2}$ choices for $Z$.
        For each $z \in Z$, $Q_z \in S^z_G$, so by \cref{lem:GL:Sv}, 
        there are at most $n^{k^2(k+1)}$ choices for each $Q_z$,
        and therefore at most $n^{k^2(k+2)}$ choices for $Q$.
    \end{claimproof}
    
    Now, let $G_0 = G - (Q \cup R)$ and $S_0 = S \setminus Q \setminus R$.
    Note that 
    $S_0 \subseteq S \setminus S_u \setminus S_v \setminus S_A \setminus S_B$.
    Moreover, $A$ is a connected component of $G_0 - S_0$,
    and there is a connected component $B_0$ of $G_0 - S_0$ that contains $\tB \setminus \{u\}$. We conclude that $S_0$ is a minimal separator of $G_0$, with $A$ and $B_0$ being connected components of $G_0 - S_0$ that are full to $S_0$.
    We now show that we can use the induction hypothesis to bound the number of choices for $S_0$.
    
    \begin{claim}\label{claim:zeta}
        $\zeta_{G_0}(S_0) < \zeta_G(S)$.
    \end{claim}
    \begin{claimproof}
       Let $I_0 \subseteq S_0$ be an independent set such that for all $y \in V(G_0) \setminus S_0$, $\card{N_{G_0}(y) \cap I_0} \le 1$.
        Let $I = I_0 \cup \{v\}$; $I$ is still an independent set since $S_0 \subseteq S \setminus N_G(v)$.
        We argue that for all $y \in V(G) \setminus S$, $\card{N_G(y) \cap I} \le 1$.
        Suppose that $y \in N_G(v)$. Since $S_0 \cap (S_A \cup S_B) = \emptyset$, we have that $N_G(y) \cap S_0 = \emptyset$ and therefore $|N_G(y)\cap I|=1$.
        We may now assume that $y \notin N_G(v)$.
        Suppose that $\card{N_G(y) \cap I} > 1$.
        Since $y \notin N_G(v)$, we conclude that $y \notin V(G_0) \setminus S_0$, otherwise $y$ would have at least two neighbors in $I_0$, a contradiction with the choice of $I_0$ in~$S_0$ in the graph $G_0$.
        This means that $y \in R \setminus S$,
        and therefore $y \in N_G(v) \cap B$,
        which is a contradiction with our assumption that $y\notin N_G(v)$.
    \end{claimproof}
    
    The number of choices for $u$, $v$, $Q$, and $R$ is at most $n^{2+(k^2+1)(k+2)}$, see \cref{claim:choicesQR}.
    For $S_0$, there are at most $n^{(L-1)(4+(k^2+2)(k+2))}$ choices by \cref{claim:zeta} and the induction hypothesis.
    Given $Q$, $R$, and $S_0$, there are at most $n$ choices for
    $A \in \comp(G - Q - R - S_0)$ and we obtain $S$ as $N(A)$.
    Taking into account also at most $n^{k+2}$ separators $S$ for which there
    exists $v \in V(G) \setminus S$ with $S \subseteq N(v)$, the number of separators
    of $G$ is bounded by

    \begin{align*}
    &n^{2+(k^2+1)(k+2)} \cdot n^{(L-1)(4+(k^2+2)(k+2))} \cdot n + n^{k+2} \\
      &\qquad\leq 
      n^{4+(k^2+1)(k+2)} \cdot n^{(L-1)(4+(k^2+2)(k+2))} \leq n^{L(4+(k^2+2)(k+2))}.
      \end{align*}

This completes the proof.

\section{Wrapping up the proof of \cref{thm:conj}}\label{sec:wrap-up}
To conclude the proof of \cref{thm:conj}, we observe that the following
statement essentially follows from the combinations
of Lemma~9 and the proof of Lemma~15 of~\cite{GL20}.

\begin{lemma}[\cite{GL20}]\label{lem:GL20}
If $G$ is a $k$-creature-free graph that contains a minimal separator $S$
with $\zeta_G(S) > (8k^2)^{k+2}$, then $G$ contains a $k$-skinny-ladder as an induced minor.
\end{lemma}
\begin{proof}[sketch.]
Let $G$ and $S$ be as in the lemma statement.
Let $I_0 \subseteq S$ be an independent set of size $\zeta_G(S)$ such that
no vertex $v \in V(G) \setminus S$ is adjacent to more than one vertex of $I_0$. 

Let $L_0$ and $R_0$ be two full sides of $S$. 
Lemma~9 of~\cite{GL20} asserts that there exists an induced path $L$ in $L_0$,
an induced path $R$ in $R_0$, and a set $I \subseteq I_0$ of size at least $|I_0|/k^2 > (8k^2)^{k+1}$
such that $L$ dominates $I$ and $R$ dominates $I$.

This is exactly the situation at the end of the first paragraph of the proof of Lemma~15
of~\cite{GL20}. A careful inspection of that proof shows that the remainder of the proof
(as well as the invoked Lemmata~8, 13 and~14)
do not use other assumptions of Lemma~15. 
Hence, we obtain the conclusion: a $k$-skinny ladder as an induced minor of $G$.
\end{proof}

By combining \cref{thm:main} and \cref{lem:GL20}, we obtain \cref{thm:conj}.

\section{Conclusion}
In \cref{thm:conj} we showed that if a graph class $\mathcal{G}$ excludes $k$-creatures as induced subgraphs 
and $k$-skinny ladders as induced minors, then $\mathcal{G}$ is tame.
However, note that while $k$-creatures have exponential (in $k$) number of minimal separators, this is not the case for $k$-skinny ladders: the class of $k$-skinny ladders (over all $k$) is tame.
Thus the implication reverse to the one in \cref{thm:conj} does not hold.

Observe that the full tame/feral dichotomy for arbitrary hereditary graph classes is simply
false due to some very obscure examples.
Let $H_k$ be the $(k,2^k+1)$-theta graph: $k$ paths of length $2^k+1$ with common endpoints. 
Note that $H_k$ has $2^{k^2} + 2^{\Oh(k)}$
minimal separators ($2^{k^2}$ of them choose one internal vertex on each path) and
$k2^k +2$ vertices, so the number of minimal separators of $H_k$ is around
$|V(H_k)|^{\log |V(H_k)|}$. Hence, the hereditary class of all induced subgraphs
of all graphs $H_k$ for $k \in \mathbb{N}$ is neither tame nor feral. 

However, it is still interesting to try to obtain a tighter classification between tame
and feral graph classes for some more ``well-behaved'' hereditary graph classes.
As discussed in Conjecture~4 of~\cite{GL20}, 
a good restriction that excludes artificial examples as in the previous paragraph
is to focus on induced-minor-closed graph classes.

\bibliographystyle{abbrv}
\bibliography{references}
\end{document}